\documentclass[11pt,leqno]{amsart}
\usepackage{verbatim,amssymb,amsfonts,latexsym,amsmath,amsthm,bbm,nicefrac,xspace,enumerate,tikz,marvosym,mathtools}
\usetikzlibrary{decorations.pathmorphing,shapes}

\newtheorem{theorem}{Theorem}[section]
\newtheorem*{theorem*}{Theorem}
\newtheorem*{corollary*}{Corollary}

\newtheorem{lemma}[theorem]{Lemma}
\newtheorem{corollary}[theorem]{Corollary}
\newtheorem{proposition}[theorem]{Proposition}
\newtheorem{question}[theorem]{Question}

\theoremstyle{definition}

\theoremstyle{remark}

\newcommand{\Q}{\mathbb{Q}}

\newcommand{\N}{\mathbb{N}}

\newcommand{\A}{\mathcal{A}}
\newcommand{\B}{\mathcal{B}}

\renewcommand{\P}{\mathcal{P}}

\newcommand{\U}{\mathcal{U}}
\newcommand{\V}{\mathcal{V}}

\newcommand{\explicitSet}[1]{\left\lbrace #1 \right\rbrace}
\newcommand{\brackets}[1]{\left\langle #1 \right\rangle}
\newcommand{\set}[2]{\explicitSet{#1 \colon #2}}
\newcommand{\seq}[2]{\brackets{#1 \colon #2}}
\newcommand{\<}{\langle}
\renewcommand{\>}{\rangle}
\renewcommand{\a}{\alpha}
\renewcommand{\b}{\beta}

\newcommand{\dlt}{\delta}

\newcommand{\s}{\sigma}

\newcommand{\w}{\omega}
\newcommand{\0}{\emptyset}

\newcommand{\sub}{\subseteq}
\newcommand{\rest}{\!\restriction\!}

\newcommand{\card}[1]{\left\lvert #1 \right\rvert}

\newcommand{\PP}{\mathbb{P}}
\newcommand{\forces}{\Vdash}

\newcommand{\pwmf}{\nicefrac{\mathcal{P}(\w)}{\mathrm{Fin}}}

\newcommand{\pnmf}{\nicefrac{\mathcal{P}(\N)}{\mathrm{Fin}}}

\newcommand{\continuum}{\mathfrak{c}}

\newcommand{\dom}{\mathfrak d}
\newcommand{\bdd}{\mathfrak b}
\renewcommand{\split}{\mathfrak s}

\newcommand{\ch}{\ensuremath{\mathsf{CH}}\xspace}

\newcommand{\zfc}{\ensuremath{\mathsf{ZFC}}\xspace}
\newcommand{\pfa}{\ensuremath{\mathsf{PFA}}\xspace}
\newcommand{\ocama}{\ensuremath{\mathsf{OCA + MA}}\xspace}

\newcommand{\oca}{\ensuremath{\mathsf{OCA}}\xspace}

\newcommand{\M}{\mathbb{M}}
\newcommand{\dLP}{\diamondsuit(\mathsf{LP})}
\newcommand{\dLPm}{\diamondsuit^-(\mathsf{LP})}
\newcommand{\cLP}{\dagger(\mathsf{LP})}

\begin{document}

%%%%%%%%%%%%
\title[$\dLP$ and nontrivial automorphisms of $\mathcal P(\omega)/\mathrm{Fin}$]{A parametrized $\diamondsuit$ for the Laver property \\ and nontrivial automorphisms of $\mathcal P(\omega)/\mathrm{Fin}$}
%%%%%%%%%%%%
\author{Will Brian}
\address {
W. R. Brian\\
Department of Mathematics and Statistics\\
University of North Carolina at Charlotte\\
%9201 University City Blvd.\\
Charlotte, NC %28223, 
(USA)}
\email{wbrian.math@gmail.com}
\urladdr{wrbrian.wordpress.com}
%%%%%%%%%%%%
\author{Alan Dow}
\address {
A. S. Dow\\
Department of Mathematics and Statistics\\
University of North Carolina at Charlotte\\
%9201 University City Blvd.\\
Charlotte, NC %28223, 
(USA)}
\email{adow@charlotte.edu}
\urladdr{https://webpages.uncc.edu/adow}
%%%%%%%%%%%%

%%%%%%%%%%%%
%\subjclass[2020]{03E35, 05C90, 06E25, 08A35, 37B99, 54D40}
%%%%%%%%%%%%

%%%%%%%%%%%%
\thanks{The first author is supported in part by NSF grant DMS-2154229}
%%%%%%%%%%%%

%%%%%%%%%%%%
\begin{abstract}
We introduce a new parametrized diamond principle denoted $\diamondsuit(\mathsf{LP})$. 
This principle is akin to the parametrized diamonds of Moore, Hru\v{s}\'ak, and D\v{z}amonja, each of which corresponds to some cardinal invariant of the continuum, and gives a $\diamondsuit$-like guessing principle implying the corresponding invariant is $\aleph_1$. Our principle $\diamondsuit(\mathsf{LP})$ is a $\diamondsuit$-like guessing principle implying the Laver property holds over a given inner model, such as the ground model in a forcing extension. 

We show $\diamondsuit(\mathsf{LP})$ holds in many familiar models of $\mathsf{ZFC}$ obtained by forcing, namely those obtained from a model of $\mathsf{CH}$ by a length-$\omega_2$ countable support iteration of proper Borel posets with the Laver property. This is true for essentially the same reason that the usual parametrized diamonds hold in similarly described forcing extensions where their corresponding cardinal invariant is $\aleph_1$. 

We also prove that if $\diamondsuit(\mathsf{LP})$ holds over an inner model of $\mathsf{CH}$ then there are nontrivial automorphisms of $\mathcal P(\omega)/\mathrm{Fin}$; in fact we get particularly nice automorphisms extending nontrivial involutions built around $P$-points in the ground model. 
Additionally, we show that, like the Sacks model, all automorphisms of $\mathcal P(\omega)/\mathrm{Fin}$ are somewhere trivial in the Mathias model. This puts a limitation on the kinds of automorphisms obtainable from $\diamondsuit(\mathsf{LP})$. 
\end{abstract}
%%%%%%%%%%%%

%%%%%%%%%%%%
\maketitle
%%%%%%%%%%%%

%%%%%%%%%%%%
\section{Introduction}
%%%%%%%%%%%%

An automorphism of the Boolean algebra $\pwmf$ is called \emph{trivial} if it is induced by a function on $\w$.  
Walter Rudin proved 70 years ago in \cite{Rudin} that \ch implies there are a total of $2^\continuum$ automorphism of $\pwmf$. As there are only $\continuum$ functions $\w \to \w$, hence only $\continuum$ trivial automorphisms of $\pwmf$, this means \ch implies there are nontrivial automorphisms. 

Years later, Shelah proved in \cite[Chapter 4]{Shelah} that it is consistent with \zfc that all automorphisms of $\pwmf$ are trivial. 
Building on Shelah's work, Shelah and Stepr\={a}ns showed in \cite{SS} that \pfa implies all automorphisms are trivial. Veli\v{c}kovi\'c showed in \cite{Velickovic} that \ocama suffices, and that it is consistent with $\mathsf{MA}_{\aleph_1}$ to have nontrivial automorphisms. (Here $\mathsf{OCA}$ denotes Todor\v{c}evi\'c's Open Coloring Axiom, defined in \cite{Todorcevic}, sometimes denoted $\mathsf{OCA_T}$ or $\mathsf{OGA}$.) 
%We will say more later on Veli\v{c}kovi\'c's model of $\mathsf{MA}_{\aleph_1}$ with nontrivial automorphisms. 
Building on work of Moore in \cite{Moore}, De Bondt, Farah, and Vignati showed recently in \cite{DFV} that \oca alone implies all automorphisms of $\pnmf$ are trivial. 
Further models in which all automorphisms of $\pwmf$ are trivial, models not satisfying $\mathsf{OCA}$,  have also been constructed. These include models with $\continuum > \aleph_2$ in \cite{Dow} or with $\dom = \aleph_1$ in \cite{FarSh}. (Recall that \oca implies $\dom = \continuum = \aleph_2$).

On the other hand, Veli\v{c}kovi\'c outlines an argument in \cite{Vel}, which he attributes to Baumgartner, showing that the existence of a $P$-point with character $\aleph_1$ implies the existence of a nontrivial automorphism (see also the second half of \cite{SS1}). In particular, the existence of nontrivial automorphisms is consistent with $\neg$\ch. 
Furthermore, Shelah and Stepr\={a}ns prove in \cite{SS1} that there are nontrivial automorphisms of $\pwmf$ in the $\aleph_2$-Cohen model. In fact, their argument shows that after forcing over a model of \ch to add $\aleph_2$ Cohen reals, every automorphism from the ground model extends to an automorphism in the forcing extension. 
The same authors also show in \cite{SS3} that there are nontrivial automorphisms if certain cardinal invariants related to $\dom$ are equal to $\aleph_1$. 

All this raises a few questions: 
\begin{itemize}
\item[$\circ$] What familiar set-theoretic axioms imply that all automorphisms of $\pwmf$ are trivial? 
\item[$\circ$] Or, to ask essentially the same question in a different way, in which of the familiar models of \zfc are all automorphisms of $\pwmf$ trivial?
\item[$\circ$] Furthermore, what kinds of nontrivial automorphisms can be constructed from familiar axioms? What kinds of nontrivial automorphisms exist in familiar models?
\end{itemize}
There are several ways to interpret the third question. Here we focus on basic topological properties of automorphisms (are they of finite order? somewhere trivial? what does the set of fixed points look like?) 
Another interesting interpretation comes from considering properties defined by the action of an automorphism on the trivial automorphisms (see \cite{WBshift,BF}). 

The objective of this paper is to contribute to the first and third of these questions. 
It is worth pointing out from the start that, strictly speaking, this paper does not contribute any new answers to the second question. 
All the nontrivial automorphisms of $\pwmf$ constructed here are constructed from a new combinatorial principle introduced below, denoted $\dLP$. 
We prove that $\dLP$ holds in models obtained from models of $\mathsf{CH}$ by a length-$\omega_2$ countable support iteration of proper Borel posets with the Laver property (e.g., the Laver, Mathias, Miller, Sacks, and Silver models). 
These are currently the only known nontrivial models in which $\dLP$ holds. 
Shelah and Stepr\={a}ns proved already in \cite{SS3} that there are nontrivial automorphisms in such models. (They define some new cardinal characteristics, which are $\aleph_1$ in these models, and they prove that when these cardinals are $\aleph_1$ then there are nontrivial automorphisms.) The main purposes of the paper are: 
\begin{itemize}
\item[$\circ$] The introduction of the parametrized $\diamondsuit$ principle $\dLP$, which is interesting in its own right. Aside from defining the principle and using it to construct nontrivial automorphisms of $\pwmf$, we also introduce two weaker forms of $\dLP$ and articulate several open questions concerning these three new principles. 
\item[$\circ$] The simplicity of the automorphisms built from $\dLP$. 
We mentioned earlier the Baumgartner-Veli\v{c}kovi\'c construction that builds a nontrivial automorphism from a $P$-point of character $\aleph_1$. 
Their automorphisms are involutions that fix precisely one ultrafilter (the $P$-point), and are trivial on every set in the dual ideal of that ultrafilter. 
Our construction generalizes this one, using $\dLP$ to build a nontrivial involution that fixes a nowhere dense $P$-filter of character $\aleph_1$, and is trivial on every set in the dual ideal of that filter. 
\end{itemize}

%In addition to the Baumgartner-Veli\v{c}kovi\'c construction from a $P$-point, Veli\v{c}kovi\'c introduces a forcing poset in \cite{Vel}, which he uses to extend a model of \pfa to a model of \ma containing a nontrivial automorphism. The automorphism added by this poset is quite similar to the one constructed from a $P$-point of character $\aleph_1$, 

In addition to these two main objectives, in the final section of the paper we show that all automorphisms of $\pwmf$ are somewhere trivial in the Mathias model. 
The same is already known for the Sacks model (see \cite{SS2}). 
Because $\dLP$ holds in the Mathias and Sacks models, this puts a limitation on the kinds of automorphisms that can be obtained from $\dLP$: one cannot build a nowhere trivial automorphism from $\dLP$.  
While of course this fact can be deduced from the Sacks model alone, we have included a proof of the analogous fact in the Mathias model because it seems a worthwhile addition to the third question above.

%%%%%%%%%%%%
\section{Bowties of functions}
%%%%%%%%%%%%

%As mentioned in the introduction, a construction outlined by Veli\v{c}kovi\'c in \cite{Vel}, originally due to Baumgartner, builds a nontrivial automorphism from a $P$-point of character $\aleph_1$. 
%Veli\v{c}kovi\'c introduces a forcing poset in \cite{Velickovic}, which he uses to extend a model of \pfa to a model of \ma containing a nontrivial automorphism. The automorphism added by this poset closely resembles the one constructed from a $P_{\aleph_1}$-point in \cite{Vel}, but genericity arguments replace a simple recursion to accomplish the same thing with a $P_{\aleph_2}$-point in a model of \ma. 
%The purpose of this section is to extract from these constructions the general topological structure behind both constructions.
The aim of this section is to present a general topological structure from which one can build a nontrivial automorphism of $\pwmf$. The argument directly generalizes the Baumgartner-Veli\v{c}kovi\'c construction in \cite{Vel} building an automorphism from a $P$-point of character $\aleph_1$. But the version presented here applies more broadly: we show that such structures (and their accompanying automorphisms) follow from $\dLP$, which holds in the Laver model (where there are no $P$-points of character $\aleph_1$) and the Silver model (where there are no $P$-points at all \cite{CG}).  

The power set of $\w$, $\P(\w)$, when ordered by inclusion, is a complete Boolean algebra. 
The Boolean algebra $\pwmf$ 
is the quotient of this Boolean algebra by the ideal of finite sets. 
For $A,B \sub \N$, we write $A =^* B$ to mean that $A$ and $B$ differ by finitely many elements, i.e. $\card{(A \setminus B) \cup (B \setminus A)} < \aleph_0$. 
The members of $\pwmf$ are equivalence classes of the $=^*$ relation, and we denote the equivalence class of $A \sub \w$ by $[A]_\text{Fin}$. 
The Boolean-algebraic relation $[A]_\text{Fin} \leq [B]_\text{Fin}$ is denoted on the level of representatives by writing $A \sub^* B$: that is, $A \sub^* B$ means $[A]_\text{Fin} \leq [B]_\text{Fin}$, or (equivalently) $B$ contains all but finitely many members of $A$. 

The Stone space of $\P(\w)$ is $\b\w$, the \v{C}ech-Stone compactification of the countable discrete space $\w$. The Stone space of $\pwmf$ is $\w^* = \b\w \setminus \w$. 

An \emph{almost permutation} on $\w$ is a bijection between cofinite subsets of $\w$. 
Every permutation of $\w$ is an almost permutation, but not every almost permutation is a permutation or even the restriction of one (e.g. consider the successor function $n \mapsto n+1$). 
Every almost permutation $f$ of $\w$ induces an automorphism $\alpha_f$ of $\pwmf$, defined by the formula 
$$\alpha_f \big( [A]_{\mathrm{Fin}} \big) \,=\, \big[ f[A] \big]_{\mathrm{Fin}}.$$ 
An automorphism of $\pwmf$ that is induced in this way by an almost permutation of $\w$ is called \emph{trivial}. 
Note that if $f$ is not an almost permutation of $\w$ then the formula above does not define an automorphism of $\pwmf$ (though it still defines a homomorphism). 
A \emph{trivial autohomeomorphism} of $\w^*$ means the Stone dual of a trivial automorphism of $\pwmf$. Specifically, if $f$ is an almost permutation of $\w$, it induces the trivial autohomeomorphism $f^*$ defined by taking 
$f^*(u)$ to be the ultrafilter generated by $\set{f^{-1}[A]}{A \in u}$. 

If $A,B$ are infinite subsets of $\w$, then $\nicefrac{\P(A)}{\mathrm{Fin}}$ and $\nicefrac{\P(B)}{\mathrm{Fin}}$ are both isomorphic to $\pwmf$. 
As with automorphisms of $\pwmf$, an isomorphism $\nicefrac{\P(A)}{\mathrm{Fin}} \to \nicefrac{\P(B)}{\mathrm{Fin}}$ is \emph{trivial} if it is induced by the action of an almost bijection $A \to B$. 
An automorphism of $\pwmf$ is \emph{trivial on $A \sub \w$} if its restriction to $\nicefrac{\P(A)}{\mathrm{Fin}}$ is trivial in this sense. 
An automorphism is \emph{somewhere trivial} if it is trivial on some infinite $A \sub \w$. 

Extending the $=^*$ relation from subsets of $\w$ to functions, 
$f =^* g$ means $\mathrm{dom}(f) =^* \mathrm{dom}(g) =^* \set{x \in \mathrm{dom}(f) \cap \mathrm{dom}(g)}{f(x)=g(x)}$. 
Similarly, if $f$ and $g$ are functions into $\w$ then 
$f \leq g$ means $\mathrm{dom}(f) \sub \mathrm{dom}(g)$ and $f(x) \leq g(x)$ for all $x \in \mathrm{dom}(f)$, and 
$f \leq^* g$ means $\mathrm{dom}(f) \sub^* \mathrm{dom}(g)$ and $f(x) \leq g(x)$ for all but finitely many $x \in \mathrm{dom}(f) \cap \mathrm{dom}(g)$. 

We now introduce the main object of this section, bowties of functions. In addition to generalizing the Baumgartner-Veli\v{c}kovi\'c argument mentioned above, these bowties are related to the bowtie points studied by the second author and Shelah in \cite{DS2} (also known as tie-points or butterfly points). We show in this section that the existence of a bowtie of functions implies there is a (particularly nice) nontrivial automorphism of $\pwmf$. In the next section we show that $\dLP$ implies the existence of a bowtie of functions. 

For each $\a < \w_1$, let $D_\a \sub \w$ and let $f_\a$ be a permutation $D_\a \to D_\a$ with no fixed points. 
The sequence $\seq{f_\a}{\a < \w_1}$ is a \emph{bowtie of functions} if 
\begin{itemize}
%\item[$\circ$] each $h_\a$ has no fixed points, 
\item[$\circ$] $D_\a \sub^* D_\b$ and $D_\b \setminus D_\a$ is infinite whenever $\a < \b < \w_1$, 
\item[$\circ$] $f_\a =^* f_\b \rest D_\a$ whenever $\a \leq \b < \w_1$, 
\item[$\circ$] for every $A \sub \w$ there is some $\dlt_A < \w_1$ such that if $\dlt_A \leq \a < \w_1$ then $( D_\a \setminus D_{\dlt_A} ) \cap A$ and 
$( D_\a \setminus D_{\dlt_A} ) \setminus A$ 
are almost $f_\a$-invariant sets, by which we mean that 
$f_\a[ ( D_\a \setminus D_{\dlt_A} ) \cap A ] =^* ( D_\a \setminus D_{\dlt_A} ) \cap A$ 
and 
$f_\a[ ( D_\a \setminus D_{\dlt_A} ) \setminus A ] =^* ( D_\a \setminus D_{\dlt_A} ) \setminus A$.
\end{itemize}
If furthermore each $f_\a$ is an involution (meaning $f_\a^2 = \mathrm{id}$), then $\seq{f_\a}{\a < \w_1}$ is a \emph{bowtie of involutions}. 

\begin{theorem}\label{thm:bowties}
If there is a bowtie of functions then there is a nontrivial automorphism $F$ of $\pwmf$. 
Furthermore, there is a closed $P$-set in $\w^*$ with character $\aleph_1$, namely $K = \bigcap_{\a < \w_1}(\w \setminus D_\a)^*$, such that $F$ is trivial on $A$ whenever $A^*$ does not meet the boundary of $K$. 
Moreover, if the bowtie of functions is a bowtie of involutions, then $F$ is an involution of $\pwmf$.
\end{theorem}

%The following lemma was proved by Frol\'ik in \cite{Frolik}.
%
%\begin{lemma}\label{lem:FixedPoints}
%If $F: \w^* \to \w^*$ is a trivial autohomeomorphism of $\w^*$, then 
%$\mathrm{Fix}(F) = \set{u \in \w^*}{F(u)=u}$ is clopen. 
%\end{lemma}
%\begin{proof}
%This was proved by Frol\'ik in \cite{Frolik}.
%Let $f$ be an almost permutation of $\w$ such that $F = f^*$, and let $C = \set{n \in \mathrm{dom}(f)}{f(n)=n}.$ 
%We claim $\mathrm{Fix}(F) = C^*$. 
%
%It is clear from the definition of $f^*$ that if $u \in C^*$ then $f(u)=u$, so $C^* \sub \mathrm{Fix}(F)$. 
%For the less trivial direction, first observe that $f \rest (\w \setminus C)$ is an almost permutation with no fixed points. By the ``three-sets lemma'' (see \cite{EdB}), it follows that $\w \setminus C$ can be partitioned into three sets, $A_0, A_1, A_2$, such that $f^{-1}[A_i] \cap A_i = \0$ for each $i$. 
%If $u \notin C^*$ then there is some $i$ with $A_i \in u$ and $f^{-1}[A_i] \in f^*(u) = F(u)$. Because $A_i \cap f^{-1}[A_i] = \0$, this means $u \neq F(u)$. Hence $\mathrm{Fix}(F) \sub C^*$.
%\end{proof}

\begin{proof}%[Proof of Theorem~\ref{thm:bowties}]
Fix a bowtie of functions $\seq{f_\a}{\a < \w_1}$. 
For each $A \sub \w$, let $\dlt_A$ denote the least ordinal having the property in the definition of ``bowtie of functions'' above. 
Define $F: \P(\w) \to \P(\w)$ by 
$$F(A) \,=\, f_{\dlt_A}[A \cap D_{\dlt_A}] \cup ( A \setminus D_{\dlt_A} ).$$
Note that if $A =^* B$ then $\dlt_A = \dlt_B$. Consequently, $F$ induces a mapping on $\pwmf$, which we denote $\tilde F$, namely
$$\tilde F\big([A]_\mathrm{Fin}\big) \,=\, \big[F[A]\big]_\mathrm{Fin}.$$
We aim to show that this mapping is a nontrivial automorphism with the properties described in the theorem.

Fix $A \sub \w$, and fix $\dlt$ with $\dlt_A \leq \dlt < \w_1$. 
%The sets $D_{\dlt_A} \cap A$, $\big(D_\dlt \setminus D_{\dlt_A} \big) \cap A$, and $A \setminus D_\dlt$ form a partition of $A$. 
Because $\dlt_A \leq \dlt$ we have $f_{\dlt_A} =^* f_\dlt \rest D_{\dlt_A}$, hence $f_\dlt[A \cap D_{\dlt_A}] =^* f_{\dlt_A}[A \cap D_{\dlt_A}]$. Furthermore, our choice of $\dlt_A$ ensures that $f_\dlt \big[(D_\dlt \setminus D_{\dlt_A}) \cap A \big] =^* (D_\dlt \setminus D_{\dlt_A} ) \cap A$. Consequently, 
\begin{align*}
F(A) &\,=\, f_{\dlt_A}[A \cap D_{\dlt_A}] \cup ( A \setminus D_{\dlt_A} ) \\
&\,=^*\, f_\dlt[A \cap D_{\dlt_A}] \cup ( A \setminus D_{\dlt_A} ) \\
&\,=^*\, f_\dlt[A \cap D_{\dlt_A}] \cup \big( (D_\dlt \setminus D_{\dlt_A}) \cap A \big) \cup ( A \setminus D_\dlt ) \\
&\,=^*\, f_\dlt \big[ (A \cap D_{\dlt_A}) \cup \big( (D_\dlt \setminus D_{\dlt_A}) \cap A \big) \big] \cup ( A \setminus D_\dlt ) \\ 
&\,=\, f_\dlt[A \cap D_\dlt] \cup (A \setminus D_\dlt).
\end{align*}
In other words, removing the subscripted $A$'s in the definition of $F(A)$ only changes the definition modulo $\mathrm{Fin}$. 

Consider some $A,B \sub \w$, and let $\dlt = \max\{\dlt_A,\dlt_B,\dlt_{A \cup B}\}$. Then 
\begin{align*}
F(A) \cup F(B) &\,=\, \big( f_{\dlt_A}[A \cap D_{\dlt_A}] \cup ( A \setminus D_{\dlt_A} ) \big) \cup \big( f_{\dlt_B}[B \cap D_{\dlt_B}] \cup ( B \setminus D_{\dlt_B} ) \big) \\ 
&\,=^*\, \big( f_\dlt[A \cap D_\dlt] \cup ( A \setminus D_\dlt ) \big) \cup \big( f_\dlt[B \cap D_\dlt] \cup ( B \setminus D_\dlt ) \big) \\ 
&\,=\, f_\dlt[(A \cup B) \cap D_\dlt] \cup \big( (A \cup B) \setminus D_\dlt \big) \\ 
&\,=^*\, f_{\dlt_{A \cup B}}[(A \cup B) \cap D_{\dlt_{A \cup B}}] \cup \big( (A \cup B) \setminus D_{\dlt_{A \cup B}} \big) \\ 
&\,=\, F(A \cup B).
\end{align*}
A similar computation shows that $F(A \cap B) =^* F(A) \cap F(B)$, and it follows from the definition of $F$ that $F(\w) = \w$ and $F(\0) = \0$. 
This shows that $\tilde F$ is a homomorphism $\pwmf \to \pwmf$. 

To show that this homomorphism is an automorphism, it remains to show $\tilde F$ is a bijection. 
It follows straightaway from the definition of $F$ that if $A$ is infinite then $F(A)$ is also infinite, hence the kernel of $\tilde F$ is $\{[\0]_\mathrm{Fin}\}$, which implies $\tilde F$ is injective. 

For surjectivity, fix $B \sub \w$, and let $A = f_{\dlt_B}^{-1}[B \cap D_{\dlt_B}] \cup (B \setminus D_{\dlt_B})$. If $\dlt$ is such that $\max\{\dlt_A,\dlt_B\} \leq \dlt < \w_1$, then 
\begin{align*}
F(A) &\,=\, f_{\dlt_A}[A \cap D_{\dlt_A}] \cup (A \setminus D_{\dlt_A}) \\
&\,=^*\, f_\dlt[A \cap D_\dlt] \cup (A \setminus D_\dlt) 
\end{align*}
\begin{align*}
&\,=\, f_\dlt \big[ \big( h_{\dlt_B}^{-1}[B \cap D_{\dlt_B}] \cup (B \setminus D_{\dlt_B}) \big) \cap D_\dlt \big] \\ 
&\qquad \qquad \cup \big( \big( f_{\dlt_B}^{-1}[B \cap D_{\dlt_B}] \cup ( B \setminus D_{\dlt_B} ) \big) \setminus D_\dlt \big) \\ 
&\,=^*\, f_\dlt \big[ \big( f_{\dlt_B}^{-1}[B \cap D_{\dlt_B}] \cup (B \setminus D_{\dlt_B}) \big) \cap D_\dlt \big] \cup ( B \setminus D_\dlt ) \\ 
&\,=\, f_\dlt \big[ \big( f_{\dlt_B}^{-1}[B \cap D_{\dlt_B}] \cup \big( B \cap (D_\dlt \setminus D_{\dlt_B}) \big) \big] \cup ( B \setminus D_\dlt ) \\ 
&\,=\, f_\dlt \big[ \big( f_{\dlt_B}^{-1}[B \cap D_{\dlt_B}] \big] \cup f_\dlt \big[ B \cap (D_\dlt \setminus D_{\dlt_B}) \big] \cup ( B \setminus D_\dlt ) \\ 
&\,=^*\, (B \cap D_{\dlt_B}) \cup \big( B \cap (D_\dlt \setminus D_{\dlt_B}) \big) \cup ( B \setminus D_\dlt ) \\ 
&\,=\, B.
\end{align*}
As $B \sub \w$ was arbitrary, this shows $\tilde F$ is surjective, and completes the proof that $\tilde F$ is an automorphism of $\pwmf$. 

It remains to show that $\tilde F$ is not a trivial automorphism, and to prove the ``furthermore'' and ``moreover'' assertions of the theorem. To see that $F$ is nontrivial, recall that Frol\'ik proved in \cite{Frolik} that 
if $F: \w^* \to \w^*$ is a trivial autohomeomorphism of $\w^*$, then $\mathrm{Fix}(F) = \set{u \in \w^*}{F(u)=u}$ is clopen. 

To apply Frol\'ik's result, we work temporarily in the topological category. Let $\bar F: \w^* \to \w^*$ denote Stone dual of $\tilde F$, i.e., the function 
$$\bar F(u) \,=\, \textstyle \bigcup \set{\tilde F^{-1}\big( [A]_\mathrm{Fin} \big)}{A \in u}.$$
%By Lemma~\ref{lem:FixedPoints}, it suffices to show that the set of fixed points of $\bar F$ is not clopen. 
%Fix $C \sub \w$. We consider two cases. 
Let $K = \bigcap_{\a < \w_1} ( \w \setminus D_\a )^*$. 
Suppose $u \in K$ and let $A \in u$. Then 
$$F(A) \,=\, f_{\dlt_A}[A \cap D_{\dlt_A}] \cup (A \setminus D_{\dlt_A}) \,\supseteq\, A \setminus D_{\dlt_A} \in u.$$
Hence $F$ maps sets in $u$ to sets in $u$, which implies $\bar F(u) = u$. 

Now suppose $u \notin K$. Then $D_\a \in u$ for some $\a < \w_1$. Using the three set lemma (see \cite{EdB}) and the fact that $f_{\dlt_{D_\a}}$ has no fixed points, $D_\a$ can be partitioned into three sets, $A_0, A_1, A_2$, such that $f_{\dlt_{D_\a}}[A_i] \cap A_i = \0$ for each $i$. Suppose without loss of generality $A_0 \in u$. If $\dlt \geq \max\{ \a,\dlt_{D_\a},\dlt_{A_0} \}$, then  
$$F(A_0) \,=^*\, f_\dlt[A_0 \cap D_\dlt] \cup (A_0 \setminus D_\dlt) \,=^*\, f_\dlt[A_0]$$
(because $A_0 \sub D_\a \sub^* D_\dlt$), so $F$ maps a set in $u$ to a set not in $u$. This shows $u$ is not a fixed point of $\bar F$. 

Therefore $K$ is precisely the set of fixed point of $\bar F$. To see that $K$ is not clopen, fix $C \sub \w$ such that $C^* \sub P$. 
Then $D_\a \sub^* \w \setminus C$, for all $\a$, and  
if $\a < \b < \w_1$ then $(C \cup D_\b) \setminus (C \cup D_\a)$ is infinite. 
%Furthermore, $\w \setminus (C \cup D_\b) \sub^* \w \setminus (C \cup D_\a)$ whenever $\a \leq \b < \w_1$. 
Hence $\seq{\big( \w \setminus (C \cup D_\a) \big)^*}{ \a < \w_1 }$ is a decreasing sequence of nonempty clopen sets in $\w^*$. By compactness, $K \setminus C^* = \bigcap_{\a < \w_1} \big( \w \setminus (C \cup D_\a) \big)^* \neq \0$. Consequently, $C^* \neq K$. As $C^*$ was an arbitrary clopen subset of $K$, this shows that $K$ is not clopen. Applying the aforementioned result of Frol\'ik, it follows that $\bar F$ is nontrivial (hence $F$ is too).

To prove the ``furthermore'' part of the theorem, suppose that $A \sub \w$ and $A^*$ does not meet the boundary of $K$. This implies there is a partition of $A$ into two sets $B,C$ such that $B^* \sub \mathrm{Int}(K)$ and $C^* \cap K = \0$. 
The fact that $B^* \sub K$ means $B \cap D_\a =^* \0$ for all $\a < \w_1$. This implies $\dlt_B = 0$ and 
$$F(B) \,=\, f_0[B \cap D_0] \cup (B \setminus D_0) \,=^*\, B.$$ 
Hence $F \rest \nicefrac{\P(B)}{\mathrm{Fin}}$ is the identity map, and in particular it is trivial. 
As for $C$, first note that using a compactness argument as in the previous paragraph, the fact that $C^* \cap K = \0$ implies $C \sub^* D_{\dlt_0}$ for some particular $\dlt_0 < \w_1$. 
Increasing $\dlt_0$ if needed, we may (and do) assume $\dlt_0 \geq \dlt_C$. 
If $X \sub C$ and $\dlt_0 \leq \a < \w_1$, then 
$(D_\a \setminus D_{\dlt_0}) \cap X$ and $(D_\a \setminus D_{\dlt_0}) \setminus X$ are both finite, since $X \sub C \sub^* D_{\dlt_0}$. 
It follows that $\dlt_X \leq \dlt_0$. 
Consequently, if $X \sub C$ then 
$$F(X) \,=\, f_{\dlt_0}[X \cap D_{\dlt_0}] \cup (X \setminus D_{\dlt_0}) \,=\, f_{\dlt_0}[X \cap D_{\dlt_0}].$$ 
Hence $F \rest \nicefrac{\P(C)}{\mathrm{Fin}}$ is induced by $f_{\dlt_0} \rest C$ (or by $f_\dlt \rest C$ for any $\dlt \geq \dlt_0$); in particular this map is trivial. 

Finally, to prove the ``moreover'' part of the theorem, suppose each $f_\a$ is an involution. Note that $\dlt_A = \dlt_{\w \setminus A}$ for every $A \sub \w$. Then 
\begin{align*}
F\big( F(A) \big) &\,=\, f_{\dlt_{\w \setminus A}}\big[f_{\dlt_A}[A \cap D_{\dlt_A}] \big] \cup ( A \setminus D_{\dlt_{\w \setminus A}} ) \\ 
&\,=\, f_{\dlt_A}\big[f_{\dlt_A}[A \cap D_{\dlt_A}] \big] \cup ( A \setminus D_A ) \\ 
&\,=\, (A \cap D_{\dlt_A}) \cup ( A \setminus D_A ) \,=\, A.
\end{align*}
Thus if each $f_\a$ is an involution, so is $F$, and consequently $\tilde F$ as well. 
\end{proof}

Let us point out that the Baumgartner-Veli\v{c}kovi\'c argument outlined in \cite{Vel}, which shows that the existence of a $P$-point with character $\aleph_1$ implies the existence of a nontrivial automorphism, essentially shows that if there is a $P$-point with character $\aleph_1$ then there is a bowtie of involutions. 
Specifically, if $u$ is a $P$-point and $\seq{B_\a}{\a < \w_1}$ is a strictly $\sub^*$-decreasing basis for $u$, then one can define by recursion a bowtie of functions $\seq{f_\a}{\a < \w_1}$, where each $f_\a$ is an involution of $D_\a = \w \setminus B_\a$. 
(For each $A \sub \w$ the ordinal $\dlt_A$ is the least $\dlt$ such that either $B_\dlt \sub^* A$ or $B_\dlt \cap A =^* \0$. One of these options must happen for some $\dlt$ because $\set{B_\a}{\a < \w_1}$ is a basis for an ultrafilter.) 
Likewise, the bowties constructed in the next section from $\dLP$ are bowties of involutions. 
We stated the definition more generally in this section because there is some generality lost in the statement of Theorem~\ref{thm:bowties}, and no clarity gained in its proof, by restricting to involutions.

%%%%%%%%%%%%
\section{$\dLP$}
%%%%%%%%%%%%

Given a function $h: \w \to \w \setminus \{0\}$, an \emph{$h$-slalom} 
is a function $G$ with domain $\w$ such that for every $n \in \w$, $G(n) \sub \w^{n+1}$ and $|G(n)| \leq h(n)$. 
Given a function $f \in \w^\w$, we say that $G$ \emph{captures} $f$ if $f \rest (n+1) \in G(n)$ for every $n$. 
We shall only be interested in $h$-slaloms where $\lim_{n \to \infty}h(n) = \infty$, so let us define $\A = \set{h \in \w^\w}{h(n) \neq 0 \text{ for all $n$, and } \lim_{n \to \infty}h(n) = \infty}$, the set of functions admissible as the width of a slalom. 

We say that the \emph{Laver property} holds over an inner model $V$ if 
for every $f \in \w^\w$ that is bounded by some member of $V \cap \w^\w$, 
and for every $h \in V \cap \A$, 
there is an $h$-slalom $G \in V$ that captures $f$. 
A forcing poset $\PP$ is said to have the Laver property if forcing with $\PP$ produces a model that has the Laver property over the ground model. 

We now recall some definitions from \cite[Section 4]{Blass} and \cite{MHD} concerning cardinal characteristics of the continuum and their corresponding parametrized $\diamondsuit$ principles. 
An \emph{invariant} is a triple $(A,B,E)$ such that 
\begin{itemize}
\item[$\circ$] $A$ and $B$ are sets of size $\leq\mathfrak c$.
\item[$\circ$] $E \sub A \times B$ (i.e., $E$ is a relation with domain $A$ and codomain $B$).
\item[$\circ$] for every $a \in A$ there is some $b \in B$ with $(a,b) \in E$. 
\item[$\circ$] for every $b \in B$ there is some $a \in A$ with $(a,b) \notin E$. 
\end{itemize}
The \emph{evaluation} of an invariant $(A,B,E)$ is 
$$\<A,B,E\> \,=\, \min \set{\card{X}}{X \sub B \text{ and $\forall a \in A$ $\exists b \in X$ }(a,b) \in E}.$$
Most of the classical cardinal invariants studied in \cite{Blass} or \cite{vD} are naturally expressed in this form: for example $\dom = \<\w^\w,\w^\w,\leq^*\>$ and $\bdd = \<\w^\w,\w^\w,\not\geq^*\>$, $\split = \<[\w]^\w,[\w]^\w,\text{is split by}\>$ and $\mathfrak r = \<[\w]^\w,[\w]^\w,\text{does not split}\>$, etc. 

Building on the work of Hru\v{s}\'ak in \cite{Hrusak} and of Devlin and Shelah in \cite{DS}, Moore, Hru\v{s}\'ak, and D\v{z}amonja define in \cite{MHD} a guessing principle $\diamondsuit(A,B,E)$ associated to an invariant $(A,B,E)$. 
\begin{itemize}
\item[$\diamondsuit(A,B,E):\hspace{-9.75mm}$]\hspace{9.75mm} For every Borel map $F: 2^{<\w_1} \to A$ there is some $g: \w_1 \to B$ such that for every $f: \w_1 \to 2$ the set $\set{\a \in \w_1}{\big(F(f \rest \a),g(\a)\big) \in E}$ is stationary. 
\end{itemize}
In this context, a \emph{Borel} map $F: 2^{<\w_1} \to A$ in \cite{MHD} means that $A$ is a Borel subset of a Polish space, and the mapping $F \rest 2^\dlt$ is Borel for every $\dlt < \w_1$.

%Suppose $X$ and $T$ are sets, and $\g$ is an ordinal. 
%A \emph{recursion scheme} is a function $F: X^{<\g} \times T \to X$. 
%(Roughly, we can imagine a recursive construction for building a $\g$-sequence in $X$. At a given step $\a < \g$ of the recursion, we are given the sequence $\s_\a$ constructed in the earlier stages of the recursion, as well as some ``task'' to accomplish, represented by some $t_\a \in T$. Then we have a formula or algorithm, here represented by $F$, which decides from $\s_\a$ and $t_\a$ what the next member $F(\s_\a,t_\a)$ of the sequence should be.) 
%Given a function $t: \g \to T$ (an enumeration of the tasks to be accomplished in the recursion), the familiar transfinite recursion theorem states that there is a unique function $f: \g \to X$ such that 
%$$f(\a) \,=\, F\big( f \rest \a,t(\a) \big) \ \text{for every } \a < \g.$$
%In this case, let us call $f$ the \emph{result of recursing with $F$ and $t$.} 
%A little more generally, if $\b \leq \g$ then $f \rest \b$ is the \emph{result of recursing with $F$ and $t$ along $\b$.} 

Let $\mathcal S = \set{G \in \prod_{n \in \w}\P(\w^n)}{G \text{ is an $h$-slalom for some }h \in \A}.$ 
For each $G \in \mathcal S$, define $\mathsf w_G : \w \to \w$ to be the function $\mathsf w_G(n) = |G(n)|$. In other words, $\mathsf w_G$ is the $\leq$-smallest $h \in \A$ such that $G$ is an $h$-slalom (the ``$\mathsf{w}$idth'' of $G$). 

Given an inner model $V$ of (a sufficiently large fragment of) \zfc, a real $r \in \w^\w$ is called \emph{$V$-bounded} if there is some $b \in \w^\w$ with $r \leq b$ (equivalently, some $b'$ with $r \leq^* b'$). 
The guessing principle $\diamondsuit(\mathsf{LP}_V)$ is defined as follows. 

\smallskip
\begin{itemize}
\item[$\diamondsuit(\mathsf{LP}_V):\hspace{-4.5mm}$]\hspace{4.5mm} 
For every Borel mapping $F: 2^{<\w_1} \to \w^\w$ and every sequence $\seq{h_\a}{\a < \w_1}$ of functions in $\A \cap V$, 
there is a function $g: \w_1 \to \mathcal S \cap V$ such that $\mathsf w_{g(\a)} = h_\a$ for every $\a < \w_1$, 
and for every $f: \w_1 \to 2$ 
$$\qquad \ \ \ \set{\a < \w_1}{\text{if $F(f \rest \a)$ is $V$-bounded then } g(\a) \text{ captures } F(f \rest \a)}$$
is stationary. 
\end{itemize}
\smallskip

Observe that in this definition, we do not require that the sequence $\seq{h_\a}{\a < \w_1}$ is in $V$, but only that each particular $h_\a$ is in $V$. 
If the inner model $V$ is clear from context, for example in a forcing extension, then $\diamondsuit(\mathsf{LP}_V)$ is abbreviated $\diamondsuit(\mathsf{LP})$. 

There is a clear relationship between $\dLP$ and the parametrized $\diamondsuit$ principles $\diamondsuit(A,B,E)$ from \cite{MHD}. It is the principle associated to the invariant $(\w^\w,\mathcal S \cap V,\text{captures if $V$-bounded})$, with the added requirement that we can require $\mathsf w_{g(\a)} = h_\a$ for any desired sequence $\seq{h_\a}{\a < \w_1}$ in $V \cap \A$. 

The principle $\dLP$ is stronger than necessary for obtaining nontrivial automorphisms of $\pwmf$. 
The following simplified version of $\dLP$ suffices for the task at hand, and is somewhat easier to work with (hence potentially easier to apply in other contexts). 

\smallskip
\begin{itemize}
\item[$\diamondsuit^-(\mathsf{LP}_V):\hspace{-7mm}$]\hspace{7mm} 
For every sequence $\seq{h_\a}{\a < \w_1}$ in $\A \cap V$, 
there is a function $g: \w_1 \to \mathcal S \cap V$ such that $\mathsf w_{g(\a)} = h_\a$ for every $\a < \w_1$ and 
$$\qquad \ \ \ \set{\a < \w_1}{g(\a) \text{ captures } r}$$
is stationary for every $V$-bounded $r \in \w^\w$. 
\end{itemize}
\smallskip
As before, if $V$ is clear from context then $\diamondsuit^-(\mathsf{LP}_V)$ is abbreviated $\dLPm$. 

In fact, even this weakening of $\dLP$ is stronger than needed to prove the existence of nontrivial automorphisms of $\pwmf$: our proof only requires that $\set{\a < \w_1}{g(\a) \text{ captures } r}$ be nonempty, not stationary. 
Let $\dagger(\mathsf{LP}_V)$, abbreviated $\dagger(\mathsf{LP})$ when $V$ is clear from context, denote the (seemingly) even weaker principle defined by replacing the word ``stationary'' with the word ``nonempty'' in the definition of $\diamondsuit^-(\mathsf{LP}_V)$.

\begin{proposition}
$\dLP$ implies $\dLPm$, and $\dLPm$ implies $\cLP$, and $\cLP$ implies that the Laver property holds over $V$. 
\end{proposition}
\begin{proof}
Suppose $\dLP$ holds. 
Let $F_0$ be a Borel surjection $2^\w \to \w^\w$, and define $F: 2^{<\w_1} \to \w^\w$ by setting $F(\s) = F_0(\s \rest \w)$ whenever $|\s| \geq \w$, and defining $F(\s)$ arbitrarily if $|\s| < \w$. 
Let $\seq{h_\a}{\a < \w_1}$ be an arbitrary sequence of functions in $\A \cap V$, and 
let $g: \w_1 \to \mathcal S \cap V$ be a function witnessing $\dLP$ for $F$ and $\seq{h_\a}{\a < \w_1}$. 
For every $V$-bounded $r \in \w^\w$, there is a function $f_r: \w \to 2$ with $F_0(f_r) = r$. 
If $f: \w_1 \to 2$ is any function with $f \rest \w = f_r$, then $F(f \rest \a) = r$ for all $\a \in [\w,\w_1)$, and therefore $\set{\a < \w_1}{g(\a) \text{ captures } r}$ is stationary. As this is true for every $V$-bounded $r \in \w^\w$, $g$ witnesses $\dLPm$ for $\seq{h_\a}{\a < \w_1}$. 

That $\dLPm$ implies $\cLP$ is clear from their definitions. 

To prove the final assertion of the proposition, suppose $\cLP$ holds. Let $r \in \w^\w$ be $V$-bounded, and let $h \in \A \cap V$. Applying $\cLP$ with the constant sequence $h_\a = h$, there is a function $g: \w_1 \to \mathcal S \cap V$ such that $\set{\a < \w_1}{g(\a) \text{ captures } r}$ nonempty and $\mathsf{w}_{g(\a)} = h$ for all $\a$. Hence there is an $h$-slalom in $V$ that captures $r$. 
\end{proof}

\begin{question}
Can any of these implications be reversed?
\end{question}

Like with the parametrized $\diamondsuit$ principles of \cite{MHD}, which hold in sufficiently nice forcing extensions where the corresponding cardinal characteristic is $\aleph_1$, $\dLP$ holds in sufficiently nice forcing extensions by posets with the Laver Property over a model of \ch. The proof of this, which comes next, is a simple and relatively straightforward adaptation of the (interesting and quite lovely) proof found in \cite{MHD}. 

Let $\P(2)$ denote the Boolean algebra with $2$ elements, and let $\P(2)^+$ denote the corresponding notion of forcing. Many of the ``usual'' Borel forcing notions $\Q$ (Cohen, random, Laver, Mathias, Miller, Sacks, Silver, etc.) are equivalent to $\P(2)^+ \times \Q$ (i.e., the lottery sum of two copies of $\Q$). 

\begin{theorem}\label{thm:MHD}
Suppose $\seq{\Q_\a}{\a < \w_2}$ is a sequence of Borel partial orders such that for each $\a < \w_2$, $\Q_\a$ is forcing equivalent to $\P(2)^+ \times \Q_\a$. 
Let $\PP_{\w_2}$ denote the length-$\w_2$ countable support iteration of this sequence. 
Assuming \ch, the following are equivalent:
\begin{enumerate}
\item $\PP_{\w_2}$ has the Laver property.
\item $\PP_{\w_2}$ forces $\cLP$. 
\item $\PP_{\w_2}$ forces $\dLPm$. 
\item $\PP_{\w_2}$ forces $\dLP$. 
\end{enumerate} 
\end{theorem}
\begin{proof}
By the previous proposition, $(4) \Rightarrow (3) \Rightarrow (2) \Rightarrow (1)$. It remains to show $(1) \Rightarrow (4)$: i.e., if $\PP_{\w_2}$ has the Laver property then it forces $\dLP$. 

The proof of this is a straightforward adaptation of the proof of Theorem 6.6 in \cite{MHD}, which spans Lemma 6.8 through Lemma 6.12 in \cite{MHD}. 
In fact, the proof of \cite[Theorem 6.6]{MHD} works almost verbatim for proving $(1) \Rightarrow (4)$ here, except that one small change needs to be made to the proof of Lemma 6.12. 
In that proof, the authors begin by fixing a sequence $\seq{b_\xi}{\xi < \w_1}$ of members of $B$ witnessing $\<A,B,E\> = \aleph_1$. 
In the present context, we are given a sequence $\seq{h_\a}{\a < \w_1}$, and we must fix for each $\a < \w_1$ a sequence $\< g_\xi^\a :\, \xi < \w_1 \>$ of $h_\a$-slaloms in $V$ witnessing the fact that every $V$-bounded $r \in \w^\w$ is captured by an $h_\a$-slalom in $V$. (This is true because we are assuming the Laver property over a model of \ch, and in fact we can take $\< g_\xi^\a :\, \xi < \w_1 \>$ simply to be an enumeration of all $h_\a$-slaloms in $V$). Then later in the proof, when defining $\dot g(\dlt)$, we should take $h_t(\xi)$ to force $\dot g(\dlt) = g_\xi^\dlt$ (rather than simply $b_\xi$) below a condition $t \in \mathcal T$ with height $\dlt$. 
\end{proof}

\begin{corollary}\label{cor:list}
$\dLP$ holds in the Laver model, the Mathias model, the Miller model, the Sacks model, and the Silver model. 
\end{corollary}

\begin{theorem}\label{thm:main}
Suppose $\dagger(\mathsf{LP}_V)$ holds with respect to an inner model $V$ with $\w_1^V = \w_1$.  
Then there is a bowtie of functions, and consequently, there is a nontrivial automorphism of $\pwmf$. 
\end{theorem}
\begin{proof}
To begin, let $\seq{I_n}{n \in \w \setminus 3}$ denote the sequence of finite intervals in $\w$ such that 
$|I_n| = n$ and $\min(I_3) = 0$ and $\min(I_{n+1}) = 1 + \max(I_n)$ for all $n$. 
(For convenience, many countably infinite sets in this proof will be indexed with $\w \setminus 3 = \{3,4,5,\dots\}$ rather than $\w$. This simplifies some formulas later on.) 
Our goal is to use $\cLP$ to define a strictly $\sub^*$-decreasing sequence $\seq{J_\a}{\a < \w_1}$ of subsets of $\w$ such that for every $\a$, 
\begin{itemize}
\item[$\circ$] $\card{J_\a \cap I_n} \geq 2$ for all $n \geq 3$.
\item[$\circ$] Let $S_\a = \set{n \geq 3}{\card{J_\a \cap I_n} > 2}$, and let $s_\a$ denote the unique increasing bijection $\w \setminus 3 \to S_\a$. Then $\card{J_\a \cap I_{s_\a(k)}} = k$.
\end{itemize}
Our bowtie of functions will then consist of the functions 
$f_\a$ 
with domain 
$$D_\a \,=\, \textstyle \bigcup \set{J_\a \cap I_n}{\card{J_\a \cap I_n} = 2} \,=\, \bigcup_{n \in \w \setminus S_\a}J_\a \cap I_n,$$
where $f_\a$ is defined so that it swaps the two elements of $J_\a \cap I_n$ for every $n \in \w \setminus S_\a$. 
Any strictly $\sub^*$-decreasing sequence $\seq{J_\a}{\a < \w_1}$ with these properties gives rise to a sequence of permutations $f_\a: D_\a \to D_\a$ satisfying the first two items in the definition of a bowtie of functions. 
The crux of the proof is using $\cLP$ to recursively build a sequence $\seq{J_\a}{\a < \w_1}$ that also satisfies the third item in the definition. 

Before applying $\cLP$ we must choose a sequence $\seq{h_\a}{\a < \w_1}$ of functions in $\A \cap V$. 
Before doing this, however, we must first construct the sets $S_\a$ and their enumeration functions, as described above, by a transfinite recursion that takes place in $V$. 
(While $S_\a$ is defined from $J_\a$ above, it is a key feature of the proof that the $S_\a$ are defined first, before the $J_\a$. Later the $J_\a$ are chosen in such a way that the above relation between $S_\a$ and $J_\a$ holds.) 
After constructing the $S_\a$, we will define each $h_\a$ from $S_\a$. 

For the base case of the recursion, let $S_0 = \w \setminus 3$. 
For the successor step, suppose $S_\a$ is given. 
Let $s_\a$ denote the unique increasing bijection $\w \setminus 3 \to S_\a$, define $s_{\a+1}: \w \setminus 3 \to \w$ by setting $s_{\a+1}(k) = s_\a(k2^k)$ for all $k \geq 3$, and let $S_{\a+1} = \mathrm{Image}(s_{\a+1})$. 
In other words, we define $S_{\a+1}$ to be the set containing the $(k2^k - 3)^\mathrm{th}$ element of $S_\a$ for every $k \geq 3$. Clearly $S_{\a+1}$ is an infinite subset of $S_\a$, and $S_\a \setminus S_{\a+1}$ is also infinite. 
For the limit step, suppose $\seq{S_\xi}{\xi < \a}$ is given for some limit ordinal $\a$, and that it is a strictly $\sub^*$-decreasing sequence of sets. 
Let $S_\a$ be an infinite pseudo-intersection of $\set{S_\xi}{\xi < \a}$, and let $s_\a$ denote the unique increasing bijection $\w \setminus 3 \to S_\a$. Shrinking $S_\a$ further if needed, we may (and do) assume that $s_\xi \leq^* s_\a$ for all $\xi < \a$. 

As mentioned already, this entire recursion takes place in the inner model $V$, and thus the result of the recursion, the sequences $\seq{S_\a}{\a < \w_1}$ and $\seq{s_\a}{\a < \w_1}$, are members of $V$ as well. 
For each $\a < \w_1$, define $h_\a: \w \to \w$ so that $h_\a(n) = 1$ if $n < 2^{s_{\a+1}(2)}$, and otherwise  
$$h_\a(n) = j \quad \Leftrightarrow \quad n \in \big[ s_{\a+1}(j),s_{\a+1}(j+1) \big).$$
Observe that each $h_\a \in \A$, because each $s_{\a+1}$ is strictly increasing, and that $\seq{h_\a}{\a < \w_1} \in V$, because $\seq{s_\a}{\a < \w_1} \in V$. 

Fix a function $g: \w_1 \to \mathcal S \cap V$ witnessing $\cLP$ for $\seq{h_\a}{\a < \w_1}$. 
The $J_\a$ are built by transfinite recursion using $g$ as a guide. 
But before giving this construction, let us first consider one way in which subsets of $\w$ can be coded by functions $\w \to \w$. 
Define $b: \w \to \w$ by setting $b(0)=b(1)=b(2)=0$, and $b(n) = 2^{|I_n|} = 2^n$ for $n \geq 3$. 
Working in $V$, there is for each $n \geq 3$ a bijection $\phi_n$ from $b(n) = \{0,1,\dots,2^n-1\}$ to $\P(I_n)$. 
Using these bijections, there is a natural homeomorphism $\Phi$ between 
$K_b = \set{a \in \w^\w}{a \leq b}$ and $\prod_{n \in \w}\P(I_n)$. 
Thus we may (and do) think of the members of $K_b$ as ``coding'' subsets of $\w$. Specifically, $A \sub \w$ is coded by the function $n \mapsto A \cap I_n$ in $\prod_{n \in \w}\P(I_n)$, which corresponds to a member of $K_b$ via $\Phi$, and a function $a \in K_b$ codes the set $\bigcup_{n \in \w}\Phi(a)(n) = \bigcup_{n \in \w}\phi_n\big(a(n)\big)$. 

We are now ready to describe the construction of $\seq{J_\a}{\a < \w_1}$. 
For the base step, let $J_0 = \w$. Note that $|J_0 \cap I_n| = |I_n| = n$ for all $n \geq 3$, and therefore $\set{n \geq 3}{|J_0 \cap I_n| > 2} = \w \setminus 3 = S_0$. 

For the successor step, suppose $J_\a \sub \w$ and $\set{n \geq 3}{|J_\a \cap I_n| > 2} = S_\a$. Let $s_\a$ denote the unique increasing bijection $\w \setminus 3 \to S_\a$, as above, and suppose furthermore that $|J_\a \cap I_{s_\a(k)}| = k$ for all $k \geq 3$. 
We shall describe $J_{\a+1}$ by determining the value of $J_{\a+1} \cap I_n$ for every $n \in \w \setminus 3$. 
First, if $|J_\a \cap I_n| = 2$ (i.e., if $n \notin S_\a$) then set $J_{\a+1} \cap I_n = J_\a \cap I_n$. 
Second, if $n \in S_\a \setminus S_{\a+1}$, then let $J_{\a+1} \cap I_n$ be the first two elements of $J_\a \cap I_n$. 
Finally, consider $n \in S_{\a+1}$. 
Recall that $s_{\a+1}$ denotes the unique increasing bijection $\w \setminus 3 \to S_{\a+1}$, and let $k = s_{\a+1}^{-1}(n)$. 
Recall that  $s_{\a+1}(k) = s_\a(k2^k)$, and $|J_\a \cap I_n| = |J_\a \cap I_{s_{\a+1}(k)}| = |J_\a \cap I_{s_\a(k2^k)}| = k2^k$. Our goal is to choose, from among the $k2^k$ elements of $J_\a \cap I_n$, some $k$ of them to be $J_{\a+1} \cap I_n$. 

Recall that $g(\a)$ is an $h(\a)$-slalom, meaning that $g(\a)(n)$ consists of $h(\a)(n)$ finite sequences in $\w^{n+1}$, and for each $\s \in h(\a)(n)$, we have $\s(n) \in \w$. If $\s(n) < 2^n$ then $\s(n)$ codes a subset of $I_n$ via $\phi_n$. 
Let 
$$\B_k = \set{\phi_n\big(\s(n)\big) \sub I_n}{\s \in g(\a)(n) \text{ and }\s(n) < 2^n}.$$ 
In other words, $g(\a)$ is telling us, via coding, a set of subsets of $I_n$. 
Furthermore, $h_{\a}(n) = h_{\a}(s_{\a+1}(k)) = k$, so 
$\B_k$ is a set of at most $k$ subsets of $I_n$. 
(While $g(\a)(n)$ has size exactly $k$, some of the $\s \in g(\a)(n)$ may have $\s(n) \geq 2^n$, in which case they do not code subsets of $I_n$ and do not contribute to $\B_k$.) 
These $k$ sets generate a partition of $I_n$ with size at most $2^k$. 
Since $|J_\a \cap I_n| = k2^k$, there are some $k$ members of $J_\a \cap I_n$ that are all in the same piece of the partition. 
Therefore we may (and do) choose $J_{\a+1} \cap I_n$ to be a size-$k$ subset of $J_\a \cap I_n$ such that $J_{\a+1} \cap I_n$ is either contained in or disjoint from every $B \in \B_k$. 
This concludes the successor step of the construction of the $J_\a$. 
It is clear that our recursive hypotheses are preserved: we have $J_{\a+1} \sub J_\a$, and $S_{\a+1} = \set{n \geq 3}{|J_\a \cap I_n| > 2}$, and $|J_{\a+1} \cap I_{s_{\a+1}(k)}| = k$ for all $k \geq 3$. 

For the limit step, fix a limit ordinal $\a < \w_1$, and recall that $S_\a$ is a pseudo-intersection of $\set{S_\xi}{\xi < \a}$, and that $s_\xi \leq^* s_\a$ for all $\xi < \a$. 
We will define $J_\a$ in two steps: first we define a set $J^0_\a$ by describing $J^0_\a \cap I_n$ for all $n \geq 3$ (like with $J_{\a+1}$ in the successor step), and then we shrink $J^0_\a$ to the desired set $J_\a$. 
As a recursive hypothesis (one clearly preserved at successor steps, where $J_{\a+1} \sub J_\a$), assume $\seq{J_\xi}{\xi < \a}$ is $\sub^*$-decreasing. 

Fix an increasing sequence $\seq{\xi_n}{n < \w}$ of ordinals below $\a$ with $\xi_0 = 0$ and $\sup_{n < \w}\xi_n = \a$. 
Next, fix a strictly increasing sequence $\seq{k_n}{n < \w}$ of natural numbers with $k_0 > 3$ 
such that 
\begin{itemize}
\item[$\circ$] $S_\a \setminus k_n \sub S_{\xi_{n+1}}$, and 
\item[$\circ$] if $i \geq k_n$ then $s_{\xi_{n+1}}(i) \leq s_\a(i)$, and 
\item[$\circ$] $J_{\xi_{n+1}} \!\setminus s_{\xi_n}(k_n) \sub J_{\xi_n}$. 
\end{itemize}
It is possible to find such a sequence because $S_\a \sub^* S_{\xi_{n+1}}$, and $s_{\xi_{n+1}} \leq^* s_\a$, and $J_{\xi_{n+1}} \sub^* J_{\xi_n}$ for all $n$, and because each $s_{\xi_n}$ is strictly increasing. 
Define $J^0_\a \sub \w$ by taking
\begin{align*}
J^0_\a \cap I_j &= J_0 \cap I_j = I_j \qquad \text{ if } \ j \in \big[ 3,s_0(k_0) \big), \\
J^0_\a \cap I_j &= J_{\xi_{n+1}} \cap I_j \qquad \qquad\hspace{-4.2mm} \text{ if } \ j \in \big[ s_{\xi_n}(k_n),s_{\xi_{n+1}}(k_{n+1}) \big), 
\end{align*}

For each $n \in \w$, the third requirement in our choice of $k_n, k_{n-1}, \dots, k_0$ ensures that if $j \in \big[ s_{\xi_n}(k_n),s_{\xi_{n+1}}(k_{n+1}) \big)$ then 
$$J^0_\a \cap I_j = J_{\xi_{n+1}} \cap I_j \sub J_{\xi_n} \cap I_j \sub J_{\xi_{n-1}} \cap I_j \sub \dots \sub J_{\xi_0} \cap I_j.$$ 
That is, $J^0_\a \cap I_j \sub J_{\xi_\ell} \cap I_j$ for all $\ell \leq n+1$ when $j \in \big[ s_{\xi_n}(k_n),s_{\xi_{n+1}}(k_{n+1}) \big)$. 
It follows that $J_{\xi_\ell} \sub^* J_\a^0$ for all $\ell \in \w$. 
Furthermore, if $\xi < \a$ then there is some $\ell \in \w$ with $\xi < \xi_\ell$, hence $J_\xi \supseteq^* J_{\xi_\ell} \supseteq^* J^0_\a$. Consequently, $J^0_\a \sub^* J_\xi$ for every $\xi < \a$. 

For each $n \in \w$, the first two requirements in our choice of $k_n$ ensure that if $j \in \big[ s_{\xi_n}(k_n),s_{\xi_{n+1}}(k_{n+1}) \big)$ and $j = s_\a(i)$ for some $i$, then $s_\a(i) = s_{\xi_{n+1}}(i')$ for some $i' \geq i$. (The fact that $s_\a(i) \in S_\a$ implies $s_\a(i) \in S_{\xi_{n+1}}$ by the first requirement, and $s_{\xi_{n+1}}(i) \leq s_\a(i)$ by the second requirement. $s_\a(i) \in S_{\xi_{n+1}}$ implies $s_\a(i) = s_{\xi_{n+1}}(i')$ for some $i'$, and then $i \leq i'$ because $s_{\xi_{n+1}}(i) \leq s_\a(i)$ and $s_{\xi_{n+1}}$ is strictly increasing.) 
In this case, we have $|J^0_\a \cap I_{s_\a(i)}| = |J_{\xi_{n+1}} \cap I_{\s_{\xi_{n+1}}(i')}| = i' \geq i$. 
As this is true for every $n$, it follows that if $s_\a(i) > s_{\xi_0}(k_0)$ then $|J^0_\a \cap I_{s_\a(i)}| > i$. If $s_\a(i) < s_{\xi_0}(k_0)$ then $|J^0_\a \cap I_{s_\a(i)}| = |I_{s_\a(i)}| = s_\a(i) \geq i$, as $s_\a$ is strictly increasing. This shows that $\set{n \geq 3}{|J_\a^0 \cap I_n| > 2} \supseteq S_\a$, and that 
$|J_\a^0 \cap I_{s_\a(i)}| \geq i$ for all $i$. 

By the previous paragraph, it is possible to shrink $J_\a^0$ to a set $J_\a \sub J_\a^0$ such that 
$\set{n \geq 3}{|J_\a^0 \cap I_n| > 2} = S_\a$ and 
$|J_\a^0 \cap I_{s_\a(i)}| = i$ for all $i$. 
By the paragraph before the last one, we also have $J_\a \sub J^0_\a \sub^* J_\xi$ for all $\xi < \a$. 
This completes the limit step of the recursion. 

We can now define a bowtie of functions from the sequence $\seq{J_\a}{\a < \w_1}$, as outlined at the beginning of the proof: for each $\a < \w_1$ let 
$$D_\a \,=\, \textstyle \bigcup \set{J_\a \cap I_n}{\card{J_\a \cap I_n} = 2} \,=\, \bigcup_{n \in \w \setminus S_\a}J_\a \cap I_n,$$
and let $f_\a$ be the permutation of $D_\a$ that swaps the two elements of $J_\a \cap I_n$ for every $n \in \w \setminus S_\a$. Clearly each $f_\a$ has no fixed points. 

Recall that the $S_\a$ and the $J_\a$ are strictly $\sub^*$-decreasing. 
This implies the $D_\a$ are strictly $\sub^*$-increasing. 
And it is clear that if $\a \leq \b$ then, because $J_\b \sub^* J_\a$ and $J_\a$ and $J_\b$ both meet each $I_n$ in at least two points, it is true for all but finitely many $n$ that if $|J_\a \cap I_n| = 2$ then $J_\a \cap I_n = J_\b \cap I_n$. It follows that $f_\a =^* f_\b \rest D_\a$. 

Thus $\seq{f_\a}{\a < \w_1}$ satisfies the first two requirements in the definition of a bowtie of functions, and it remains to show that it satisfies the third. 
Fix $A \sub \w$. 
Let $a$ be the function that codes $A$ as described above: i.e., $A = \bigcup_{n \in \w}\Phi(a)(n) = \bigcup_{n \in \w}\phi_n\big(a(n)\big)$. 
Recall that $a(n) < 2^{|I_n|} = 2^n$ for all $n \geq 3$, which means that $a$ is bounded by a function in $V$. 
Because $g$ witnesses $\cLP$ for $\seq{h_\a}{\a < \w_1}$, the set 
$\set{\a < \w_1}{g(\a) \text{ captures } a}$ is nonempty. 
Fix some $\dlt < \w_1$ such that $g(\dlt)$ captures $a$. 

Let $\dlt_A = \dlt+1$. 
At stage $\dlt_A$ of our construction, when we defined $J_{\dlt_A}$, 
for every $n = s_{\dlt_A}(k) \in S_{\dlt_A}$ we chose $J_{\dlt_A} \cap I_n$ to be a size-$k$ subset of $J_{\dlt} \cap I_n$ such that $J_{\dlt_A} \cap I_n$ is either contained in or disjoint from every $B \in \B_k$, where 
$$\B_k = \set{\phi_n\big(\s(n)\big) \sub I_n}{\s \in g(\dlt)(n) \text{ and }\s(n) < 2^n}.$$ 
Because $g(\dlt)$ captures $a$, we have $a(n) \in \set{\s(n)}{\s \in g(\dlt)(n)}$ for every $n$. 
Consequently, $A \cap I_n = \phi_n\big(a(n)\big) \in \B_k$ for every $n$. 
Thus $J_{\dlt_A} \cap I_n$ is either contained in or disjoint from $A \cap I_n$ for every $n \in S_\a$. 

Now fix $\a$ with $\dlt_A \leq \a < \w_1$. 
We have $D_{\dlt_A} \sub^* D_\a$, and furthermore, $D_\a \setminus D_{\dlt_A}$ is the infinite union of sets of the form $J_\a \cap I_n$ with $|J_\a \cap I_n| = 2$, where $n \in S_{\dlt_A} \!\setminus S_\a$. 
And for all but finitely many such $n$, $J_\a \cap I_n \sub J_{\dlt_A} \cap I_n$ (because $J_\a \sub^* J_{\dlt_A}$). By the previous paragraph, the sets $J_{\dlt_A} \cap I_n$, for $n \in S_{\dlt_A}$, were chosen so that $J_{\dlt_A} \cap I_n$ is either contained in or disjoint from $A$. 
Since $f_\a$ simply swaps these pairs of points, we see that in all but finitely many cases, $f_\a \setminus f_{\dlt_A}$ swaps two points in $A$ with each other or two points not in $A$ with each other. In other words, $D_\a \setminus D_{\dlt_A}$ is almost $f_\a$-invariant.
\end{proof}

\begin{corollary}
There is a bowtie of functions, hence a nontrivial automorphism of $\pwmf$, in the Laver model, the Mathias model, the Miller model, the Sacks model, and the Silver model.  
\end{corollary}
\begin{proof}
This follows immediately from Theorems \ref{thm:bowties}, \ref{thm:MHD}, and \ref{thm:main}. 
\end{proof}

%%%%%%%%%%%%
\section{Automorphisms in the Mathias model}
%%%%%%%%%%%%

Shelah and Stepr\={a}ns proved in \cite{SS2} that all automorphisms of $\pwmf$ in the Sacks model are somewhere trivial. In this section we prove the same for the Mathias model. 

Let $\M$ denote the Mathias forcing poset: 
$$\M \,=\, \set{(a,A)}{a \in [\w]^{<\w},\, A \in [\w]^\w, \text{ and }\max(a) < \min(A)},$$
with the order defined by setting $(b,B) \leq (a,A)$ if and only if $b \supseteq a$, $B \sub A$, and $b \sub a \cup A$. 
The \emph{Mathias model} refers to any model obtained by forcing with a length-$\w_2$ countable support iteration of $\M$ over a model of \ch. 
The main theorem of this section is: 

\begin{theorem}\label{thm:MathiasST}
In the Mathias model, every automorphism of $\pwmf$ is somewhere trivial. 
\end{theorem}

We prove this theorem mostly by piecing together facts from the literature. To begin, recall that 
a nonprincipal ultrafilter $\U$ on $\w$ is \emph{Ramsey} if for every infinite partition $\set{A_n}{n \in \w}$ of $\w$, either there is some $n$ with $A_n \in \U$, or else there is some $S \in \U$ such that $\card{S \cap A_n} \leq 1$ for all $n$. 
More generally, if $V$ and $W$ are two models of set theory with $V \sub W$, a filter $\U \in W$ is \emph{Ramsey over $V$} if for every infinite partition $\set{A_n}{n \in \w}$ of $\w$ with $\set{A_n}{n \in \w} \in V$, either there is some $n$ with $A_n \in \U$, or else there is some $S \in \U \cap V$ such that $\card{S \cap A_n} \leq 1$ for all $n$. In other words, $\U$ is Ramsey over $V$ if $\U \cap V$ satisfies the definition of a Ramsey ultrafilter within $V$ (but we do not require $\U \cap V \in V$). 

Suppose again that $V$ and $W$ are two models of set theory with $V \sub W$. 
If $X \sub \w$ and $X \in W$, then 
$\set{A \in V}{X \sub^* A \sub \w}$ 
is a filter, which we call the filter \emph{induced} by $X$ on $V$. 
If this filter is Ramsey over $V$, then we say that $X$ \emph{induces a Ramsey ultrafilter} on $V$. 
Viewing things the other way around, note that an ultrafilter $\U \in V$ is induced by a set $X \in W$ if and only if $X$ is a pseudo-intersection for $\U$ in $W$. 

If $\U$ is an ultrafilter on $\w$, then 
$$\M(\U) \,=\, \set{(a,A) \in \M}{A \in \U}.$$
This is a variant of the Mathias forcing, sometimes referred to as the Mathias forcing ``guided'' by the ultrafilter $\U$. 
If $G$ is an $\M(\U)$-generic filter over $V$, then $S = \textstyle \bigcup \set{a}{(a,A) \in G \text{ for some }A}$ 
is a pseudo-intersection for $\U$ in $V[G]$. %, meaning that $S \sub^* X$ for all $X \in \U$. 
%Note that $\U = \set{A \in V}{S \sub^* A \sub \w}$. 
In particular, if $\U$ is Ramsey in $V$, then $\M(\U)$ adds a subset of $\w$ that induces a Ramsey ultrafilter on $V$, namely $\U$. 
%Also note that $\U$ does not generate an ultrafilter on $\w$ in $V[G]$: we say that $\M(\U)$ \emph{diagonalizes} $\U$ by adding a pseudo-intersection. 

Recall that forcing with $\pwmf$ adds a Ramsey ultrafilter on $\w$. Specifically, if $G$ is a $\pwmf$-generic filter over $V$, then 
$\U = \set{A \sub \w}{[A]_\mathrm{Fin} \in G}$ 
is a Ramsey ultrafilter on $\w$ in $V[G]$ (i.e., $G$ itself is a Ramsey ultrafilter, modulo the natural transfer from $\pwmf$ to $\P(\w)$). 
The following fact was proved by Mathias in \cite{Mathias}: 

\begin{lemma}[Mathias]\label{lem:Mathias}
The poset $\M$ is forcing equivalent to the $2$-step iteration $\pwmf * \M(\dot \U)$, where $\dot \U$ is a name for the generic Ramsey ultrafilter added by forcing with $\pwmf$, as described above. 
\end{lemma}
\noindent In what follows, the \emph{ultrafilter added by forcing with $\M$} refers specifically to the Ramsey ultrafilter added by the first stage of this decomposition. 
Mathias' proof of this lemma shows that if $G$ is an $\M$-generic filter over $V$, then the ultrafilter added by forcing with $\M$ is definable in $V[G]$ as 
$$\U \,=\, \set{B \sub \w}{\text{there is some } (a,A) \in G \text{ with } A \sub^* B}.$$
Equivalently, this is the filter $\U_X$ induced on $V$ by the $\M$-generic real 
$$X \,=\, \textstyle \bigcup \set{a \in [\w]^{<\w}}{\text{there is some $A$ with } (a,A) \in G}.$$
In particular, the ultrafilter $\U$ added by forcing with $\M$ has a pseudo-intersection in $V[G]$, and $\U$ is induced on $V$ by this pseudo-intersection. 

We shall also need the following two results proved by Shelah and Spinas in \cite{ShelahSpinas}. 
These results are stated ``from the ground model'' in \cite{ShelahSpinas}, but we prefer to state them ``from the extension'' instead. 
For the statement of both results, suppose $V \models \ch$, let $\M_{\w_2}$ denote the length-$\w_2$ countable support iteration of $\M$, and let $G$ be an $\M_{\w_2}$-generic filter over $V$. 
For each $\a < \w_2$, let $\M_\a$ denote the first $\a$ stages of the iteration, and let $G_\a = G \restriction \M_\a$. 

%We shall also need the following result proved by Shelah and Spinas in \cite{ShelahSpinas}. 
%This result is stated ``from the ground model'' in \cite{ShelahSpinas}, but we prefer to state it ``from the extension'' instead. 
%For the statement of the result, suppose $V \models \ch$, let $\M_{\w_2}$ denote the length-$\w_2$ countable support iteration of $\M$, and let $G$ be an $\M_{\w_2}$-generic filter over $V$. 
%For each $\a < \w_2$, let $\M_\a$ denote the first $\a$ stages of the iteration, and let $G_\a = G \restriction \M_\a$. 

Recall that $C \sub \w_2$ is an $\w_1$\emph{-club} if any increasing $\w_1$-sequence in $C$ has its limit in $C$.

\begin{lemma}[Shelah-Spinas]\label{lem:ShSp1}
There is an $\w_1$-club $C \sub \w_2$ such that the following holds for all $\a \in C$: 
If $A \sub \w$ induces a Ramsey ultrafilter on $V[G_\a]$, then there is some $A' \in V[G_{\a+1}]$ such that $A$ and $A'$ induce the same ultrafilter on $V[G_\a]$. 
\end{lemma}
\begin{proof}
This is a paraphrase of Proposition 2.3 in \cite{ShelahSpinas}. 
\end{proof}

Recall that two filters $\U$ and $\V$ on $\w$ are \emph{RK-equivalent} if there is a bijection $h: \w \to \w$ such that $A \in \U$ if and only if $h[A] \in \V$. 
When this is the case, we say that $\U$ and $\V$ are RK-equivalent \emph{via the function $h$}.

\begin{lemma}[Shelah-Spinas]\label{lem:ShSp2}
Let $\U \in V[G_{\a+1}]$ denote the Ramsey ultrafilter added by forcing with $\M$ over $V[G_\a]$ at stage $\a$ of the iteration. 
Suppose $A \sub \w$ and $A \in V[G_{\a+1}]$. 
If $A$ induces a Ramsey ultrafilter $\mathcal V$ on $V[G_\a]$, then $\mathcal U$ and $\mathcal V$ are RK-equivalent via some $h \in V[G_\a]$. 
\end{lemma}
\begin{proof}
This is a paraphrase of Proposition 2.4 in \cite{ShelahSpinas}. 
\end{proof}

For the remainder of this section, we abuse notation slightly by letting $F(A)$ denote some (any) representative of the equivalence class  $F\big( [A]_{\mathrm{Fin}} \big)$ whenever $F$ is an automorphism of $\pwmf$ and $A \sub \w$. We also write $F(\U)$, rather than $F[\U]$, for the image of a filter $\U$ under $F$.

\begin{lemma}\label{lem:RamseyFilter}
Suppose $F$ is an automorphism of $\pwmf$, and let $\dot \U$ be the canonical name for the generic ultrafilter added by forcing with $\pwmf$. Then $\forces_{\pwmf} F(\dot \U)$ is Ramsey.
\end{lemma}
\begin{proof}
This proof is essentially identical to the usual proof that $\forces_{\pwmf} \dot \U$ is Ramsey. 
Let $G$ be $\pwmf$-generic over $V$, and let $\U = \set{A \sub \w}{[A]_\mathrm{Fin} \in G}$ denote the generic ultrafilter added by forcing with $\pwmf$. 
Because $\pwmf$ is a countably closed forcing, it adds no new reals, and in particular it adds no new partitions of $\w$. 

Working in the ground model, let $\set{A_n}{n \in \w}$ be an infinite partition of $\w$ and fix a condition $[A]_\mathrm{Fin} \in \pwmf$. 
If there is some $n$ such that $F(A) \cap A_n$ is infinite, then $[A \cap F^{-1}(A_n)]_\mathrm{Fin} \leq [A]_\mathrm{Fin}$ and $[A \cap F^{-1}(A_n)]_\mathrm{Fin} \forces A_n \in F(\dot \U)$. 
On the other hand, if $F(A) \cap A_n$ is finite for all $n$, then there is an infinite $S \sub F(A)$ such that $|S \cap A_n| \leq 1$ for all $n$; then $[F^{-1}(S)]_\mathrm{Fin} \leq [A]_\mathrm{Fin}$ and $[F^{-1}(S)]_\mathrm{Fin} \forces S \in F(\dot \U)$. 
Either way, there is an extension of $[A]_\mathrm{Fin}$ forcing ``either there is some $n$ with $A_n \in F(\dot \U)$, or else there is some $S \in \dot \U$ such that $\card{S \cap A_n} \leq 1$ for all $n$.''

Hence, in the extension, either there is some $n$ with $A_n \in F(\U)$, or else there is some $S \in F(\U)$ such that $\card{S \cap A_n} \leq 1$ for all $n$. As this is true for an arbitrary partition $\set{A_n}{n \in \w}$ in $V$, and because every partition of $\w$ in $V[G]$ is already in $V$, this shows $F(\U)$ is Ramsey in $V[G]$. 
\end{proof}

\begin{lemma}\label{lem:RKinequivalence1}
Suppose $F$ is an automorphism of $\pwmf$, and $A \in [\w]^\w$, and $h$ is an almost bijection from $A$ to $F(A)$. 
If $F \rest \nicefrac{\mathcal P(A)}{\mathrm{Fin}}$ is not induced by $h$, then there is an infinite $B \sub A$ such that $F(B) \cap h[B]$ is finite. 
\end{lemma}
\begin{proof}
Assuming $F \rest \nicefrac{\mathcal P(A)}{\mathrm{Fin}}$ is not induced by $h$, there is some infinite $C \sub A$ such that $F(C) \triangle h[C] = (F(C) \setminus h[C]) \cup (h[C] \setminus F(C))$ is infinite. 
If $h[C] \setminus F(C)$ is infinite, then let $B = C \setminus h^{-1}[F(C)]$. 
If $F(C) \setminus h[C]$ is infinite, then let $B = h^{-1}[F(C) \setminus h[C]]$. 
\end{proof}

The following lemma can be considered implicit in \cite{GoldsternShelah} or \cite{ShelahSpinas}, though it is not explicitly stated in either. 

\begin{lemma}\label{lem:RKinequivalence2}
Suppose $F$ is a nowhere trivial automorphism of $\pwmf$, and let $\dot \U$ be a name for the generic ultrafilter added by forcing with $\pwmf$. Then $\forces_{\pwmf}$ $\,\dot \U$ and $F(\dot \U)$ are not RK-equivalent. 
\end{lemma}
\begin{proof}
Let $G$ be $\pwmf$-generic over $V$, and let $\U = \set{A \sub \w}{[A]_\mathrm{Fin} \in G}$ denote the generic ultrafilter added by forcing with $\pwmf$. 
Because $\pwmf$ is a countably closed forcing, it adds no new reals, and in particular it adds no new bijections $\w \to \w$, so if $h \in V[G]$ is a bijection $\w \to \w$ then $h \in V$. 
Thus, in order to show $\U$ and $F(\U)$ are not RK-equivalent in $V[G]$, it suffices to show they are not RK-equivalent via a bijection $h \in V$. 

Working in the ground model, fix a bijection $h: \w \to \w$ and fix a condition $[A]_\mathrm{Fin} \in \pwmf$. 
As we are assuming $F$ is nowhere trivial, $F \rest \nicefrac{\P(A)}{\mathrm{Fin}}$ is not induced by $h \rest A$. 
By Lemma~\ref{lem:RKinequivalence1}, there is an infinite $B \sub A$ such that $F(B) \cap h[B]$ is finite. So $[B]_\mathrm{Fin} \leq [A]_\mathrm{Fin}$, and 
$$[B]_\mathrm{Fin} \,\forces\, B \in \dot \U,\, F(B) \in F(\dot \U) \text{, and consequently, } h[B] \notin F(\dot \U).$$
Thus an arbitrary condition $[A]_\mathrm{Fin} \in \pwmf$ has an extension forcing $\dot \U$ and $F(\dot \U)$ are not RK-equivalent via $h$. 

Hence $\U$ and $F(\U)$ are not equivalent via $h$ in $V[G]$. As $h$ was arbitrary bijection in $V$, and every bijection in $V[G]$ is already in $V$, this shows $\U$ and $F(\U)$ are not RK-equivalent in $V[G]$.
\end{proof}

\begin{proof}[Proof of Theorem~\ref{thm:MathiasST}]
Suppose $V \models \ch$. 
Let $\M_{\w_2}$ denote the length-$\w_2$ countable support iteration of $\M$, and let $G$ be a $\M_{\w_2}$-generic filter over $V$. 
For each $\a < \w_2$, let $\M_\a$ denote the first $\a$ stages of the iteration, and let $G_\a = G \restriction \M_\a$. 

Aiming for a contradiction, suppose that $F$ is a nowhere trivial automorphism of $\pwmf$ in $V[G]$. For each $\a < \w_2$, let $F_\a = F \restriction V[G_\a]$. 
By a standard reflection argument, there is an $\w_1$-club $C \sub \w_2$ such that for all $\a \in C$, $F_\a \in V[G_\a]$ and $V[G_\a] \models$ ``$F_\a$ is a nowhere trivial automorphism of $\pwmf$." 
By intersecting $C$ with the $\w_1$-club set described in Lemma~\ref{lem:ShSp1}, we may (and do) assume that if $\a \in C$ and 
$A \sub \w$ induces a Ramsey ultrafilter on $V[G_\a]$, then there is some $A' \in V[G_{\a+1}]$ such that $A$ and $A'$ induce the same ultrafilter on $V[G_\a]$.

Fix $\a \in C$. 
At stage $\a$ of our iteration, we force with $\M$ in $V[G_\a]$ to obtain $V[G_{\a+1}]$. 
By Lemma~\ref{lem:Mathias}, we may view this as a $2$-step process: $V[G_{\a+1}] = V[G_\a][H][K]$, where $H$ is $\pwmf$-generic over $V[G_\a]$ and $K$ is $\M(\U)$-generic over $V[G_\a][H]$, where $\U \in V[G_\a][H]$ denotes the Ramsey ultrafilter added by forcing with $\pwmf$. 
Working in $V[G_\a][H]$, Lemma~\ref{lem:RamseyFilter} tells us that $F(\U)$ is a Ramsey ultrafilter, and 
Lemma~\ref{lem:RKinequivalence2} tells us that $\U$ and $F(\U)$ are not RK-equivalent. 

In $V[G]$ the filter $\U$ is no longer an ultrafilter: it has a pseudo-intersection $S$, added by forcing with $\M(\U)$ over $V[G_\a][H]$. Because $F$ is an automorphism of $\pwmf$, this implies $F(S)$ should be a pseudo-intersection for $F(\U)$ as well. 
In other words, $F(\U)$ is an ultrafilter in $V[G_\a][H]$ and 
$$F(\U) \,=\, \set{X \in V[G_\a][H]}{F(S) \sub^* X \sub \w}.$$ 
Thus $F(S)$ induces a Ramsey ultrafilter on $V[G_\a]$, namely $F(\U)$. 

Because $\a \in C$, there is some $A' \sub \w$ such that $A' \in V[G_{\a+1}]$ and $A'$ induces the ultrafilter $F(\U)$ on $V[G_\a]$. 
By Lemma~\ref{lem:ShSp2}, this implies $\U$ and $F(\U)$ are RK-equivalent via some $h \in V[G_\a]$, contradicting  Lemma~\ref{lem:RKinequivalence2}. 
\end{proof}

%%%%%%%%%%%%

%%%%%%%%%%%%


\begin{thebibliography}{99}
%%%%%%%%%%%%

\bibitem{Blass} A. R. Blass, ``Combinatorial cardinal characteristics of the continuum,'' in \emph{Handbook of Set Theory}, eds. M. Foreman and A. Kanamori, Springer-Verlag (2010), pp. 395--489.
\bibitem{DFV} B. De Bondt, I. Farah, and A. Vignati, ``Trivial isomorphisms between reduced products,'' to appear in \emph{Israel Journal of Mathematics}, currently available at \texttt{https://arxiv.org/abs/2307.06731}.
\bibitem{WBshift} W. R. Brian, ``Does $\pwmf$ know its right hand from its left?'' preprint currently available at \texttt{https://arxiv.org/abs/2402.04358}.
\bibitem{BF} W. R. Brian and I. Farah, ``Conjugating trivial automorphisms of $\pnmf$,'' to appear in \emph{Transactions of the American Mathematical Society}, currently available at \texttt{https://arxiv.org/abs/2410.08789}.
\bibitem{CG} D. Chodounsk\'y and O. Guzm\'an, ``There are no $P$-points in Silver extensions,'' \emph{Israel Journal of Mathematics} \textbf{232} (2019), pp. 759--773.
\bibitem{DS} K. Devlin and S. Shelah, ``A weak version of $\diamondsuit$ which follows from $2^{\aleph_0} < 2^{\aleph_1}$,'' \emph{Israel Journal of Mathematics} \textbf{29} (1978), pp. 239--247.
\bibitem{Dow} A. Dow, ``Automorphisms of $\mathcal P(\w)/\mathrm{fin}$ and large continuum,'' preprint currently available at \texttt{https://arxiv.org/abs/2207.10557}.
\bibitem{DS2} A. Dow and S. Shelah, ``Tie-points and fixed-points in $\N^*$,'' \emph{Topology and its Applications} \textbf{155} (2008), pp. 1661--1671.
\bibitem{vD} E. K. van Douwen, ``The integers and topology," in \emph{Handbook of Set-Theoretic Topology}, eds. K. Kunen and J. E. Vaughan, North-Holland, Amsterdam, 1984.
\bibitem{EdB} P. Erd\H{o}s and N. G. de Buijn, ``A colour problem for infinite graphs and a problem in the theory of relations," \emph{Indagationes Mathematicae (Proceedings)} \textbf{54} (1951), pp. 371--373.
\bibitem{FarSh} I. Farah and S. Shelah, ``Trivial automorphisms,'' \emph{Israel Journal of Mathematics} \textbf{201} (2014), pp. 701--728.
\bibitem{Frolik} Z. Frol\'ik, ``Fixed points of maps of extremally disconnected spaces and complete Boolean algebras,'' \emph{Bulletin of the Polish Academy of Sciences, Series of Mathematics, Astronomy, and Physics} \textbf{16} (1968), pp. 269--275. 
\bibitem{GoldsternShelah} M. Goldstern and S. Shelah, ``Ramsey ultrafilters and the reaping number -- {$\mathsf{Con}(\mathfrak r < \mathfrak u$),''} \emph{Annals of Pure and Applied Logic} \textbf{49} (1990), pp. 121--142.
\bibitem{Hrusak} M. Hru\v{s}\'ak, ``Another $\diamondsuit$-like principle," \emph{Fundamenta Mathematicae} \textbf{167} (2001), pp. 277--289.
%\bibitem{JO} B. J\'onsson and P. Olin, ``Almost direct products and saturation,'' \emph{Compositio Mathematica} \textbf{20} (1968), pp. 125--132.
\bibitem{Mathias} A. R. D. Mathias, ``Happy families,'' \emph{Annals of Mathematical Logic} \textbf{12} (1977), pp. 59--111. 
%\bibitem{vanMill} J. van Mill, ``An introduction to $\beta \omega$,'' in \emph{Handbook of Set-Theoretic Topology} (1984), eds. K. Kunen and J. E. Vaughan, North-Holland, pp. 503--560.
\bibitem{Moore} J. T. Moore, ``Some remarks on the Open Coloring Axiom,'' \emph{Annals of Pure and Applied Logic} \textbf{172} (2021), paper no. 102912. 
\bibitem{MHD} J. T. Moore, M. Hru\v{s}\'ak, and M. D\v{z}amonja, ``Parametrized $\diamondsuit$ principles,'' \emph{Transactions of the American Mathematical Society} \textbf{356} (2003), pp. 2281--2306.
%\bibitem{Par} I. I. Parovi\v{c}enko, ``A universal bicompact of weight $\aleph$,'' \emph{Soviet Mathematics Doklady} \textbf{4} (1963), pp. 592--595.
\bibitem{Rudin} W. Rudin, ``Homogeneity problems in the theory of \v{C}ech compactifications,'' \emph{Duke Mathematical Journal}, \textbf{23} (1956), pp. 409--419.
\bibitem{Shelah} S. Shelah, \emph{Proper Forcing}, Lecture Notes in Mathematics, vol. \textbf{940}, Springer-Verlag, Berlin, 1982
\bibitem{ShelahSpinas} S. Shelah and O. Spinas, ``The distributivity numbers of $\pwmf$ and its square,'' \emph{Transactions of the American Mathematical Society} \textbf{352} (2000), pp. 2023--2047.
\bibitem{SS} S. Shelah and J. Stepr$\bar{\mathrm{a}}$ns, ``$\mathsf{PFA}$ implies all automorphisms are trivial,'' \emph{Proceedings of the American Mathematical Society} \textbf{104} (1988), pp. 1220--1225.
\bibitem{SS1} S. Shelah and J. Stepr$\bar{\mathrm{a}}$ns, ``Non-trivial automorphisms of $\b\N \setminus \N$ without the continuum hypothesis,'' \emph{Fundamenta Mathematicae} \textbf{132} (1989), pp. 135--141.
\bibitem{SS2} S. Shelah and J. Stepr$\bar{\mathrm{a}}$ns, ``Somewhere trivial autohomeomorphisms,'' \emph{Journal of the London Mathematical Society} \textbf{49} (1994), pp. 569--580.
\bibitem{SS3} S. Shelah and J. Stepr$\bar{\mathrm{a}}$ns, ``Non-trivial automorphisms of $\P(\N)/[\N]^{<\aleph_0}$ from variants of small dominating number,'' \emph{European Journal of Mathematics} \textbf{1}, vol. 3 (2015), pp. 534--544.
%\bibitem{Step} J. Stepr\={a}ns, ``Topological invariants in the Cohen model,'' \emph{Topology and Its Applications} \textbf{45} (1992), pp. 85--101.
\bibitem{Todorcevic} S. Todor\v{c}evi\'c, \emph{Partition problems in topology}, Contemporary Mathematics, vol. \textbf{84} (1989), Providence, RI: American Mathematical Society.
\bibitem{Vel} B. Veli\v{c}kovi\'c, ``Definable automorphisms of $\pwmf$,'' \emph{Proceedings of the American Mathematical Society} \textbf{96} (1986), pp. 130--135.
\bibitem{Velickovic} B. Veli\v{c}kovi\'c, ``$\mathsf{OCA}$ and automorphisms of $\pwmf$,'' \emph{Topology and Its Applications} \textbf{49} (1992), pp. 1--12.


\end{thebibliography}
\end{document}